\documentclass[12pt]{amsart}
\usepackage{amscd}
\usepackage{graphicx}
\def\frk{\frak}               

\def\Phi{{\frk n}}
\def\Phi{{\frk N}}
\def\opn#1#2{\def#1{\operatorname{#2}}} 
\opn\chara{char} \opn\length{\ell} \opn\pd{pd} \opn\rk{rk}
\opn\projdim{proj\,dim} \opn\injdim{inj\,dim} \opn\rank{rank}
\opn\depth{depth} \opn\grade{grade} \opn\height{height}
\opn\embdim{emb\,dim} \opn\codim{codim}

\opn\Tr{Tr} \opn\bigrank{big\,rank}
\opn\superheight{superheight}\opn\lcm{lcm}
\opn\trdeg{tr\,deg}
\opn\reg{reg} \opn\lreg{lreg} \opn\ini{in} \opn\lpd{lpd}
\opn\size{size}
\opn\div{div} \opn\Div{Div} \opn\cl{cl} \opn\Cl{Cl}
\opn\Spec{Spec} \opn\Supp{Supp} \opn\supp{supp} \opn\Sing{Sing}
\opn\Ass{Ass} \opn\Min{Min} \opn\Shad{Shadow}
\opn\Ann{Ann} \opn\Rad{Rad} \opn\Soc{Soc}
\opn\Im{Im} \opn\Ker{Ker} \opn\Coker{Coker} \opn\Am{Am}
\opn\Hom{Hom} \opn\Tor{Tor} \opn\Ext{Ext} \opn\End{End}
\opn\Aut{Aut} \opn\id{id}

\opn\nat{nat}
\opn\pff{pf}
\opn\Pf{Pf} \opn\GL{GL} \opn\SL{SL} \opn\mod{mod} \opn\ord{ord}
\opn\Gin{Gin} \opn\Hilb{Hilb}
\opn\aff{aff} \opn\con{conv} \opn\relint{relint} \opn\st{st}
\opn\lk{lk} \opn\cn{cn} \opn\core{core} \opn\vol{vol}
\opn\link{link} \opn\star{star}
\opn\gr{gr}

\def\pot#1#2{#1[\kern-0.28ex[#2]\kern-0.28ex]}

\opn\dirlim{\underrightarrow{\lim}}
\opn\inivlim{\underleftarrow{\lim}}
\def\Implies{\ifmmode\Longrightarrow \else
        \unskip${}\Longrightarrow{}$\ignorespaces\fi}
\def\implies{\ifmmode\Rightarrow \else
        \unskip${}\Rightarrow{}$\ignorespaces\fi}
\def\iff{\ifmmode\Longleftrightarrow \else
        \unskip${}\Longleftrightarrow{}$\ignorespaces\fi}

\let\:=\colon
\newtheorem{Theorem}{Theorem}[section]
\newtheorem{Lemma}[Theorem]{Lemma}
\newtheorem{Corollary}[Theorem]{Corollary}
\newtheorem{Proposition}[Theorem]{Proposition}
\newtheorem{Remark}[Theorem]{Remark}

\newtheorem{Example}[Theorem]{Example}

\newtheorem{Definition}[Theorem]{Definition}

\let\epsilon\varepsilon
\let\phi=\varphi
\let\kappa=\varkappa
\def\qed{\ifhmode\textqed\fi
      \ifmmode\ifinner\quad\qedsymbol\else\dispqed\fi\fi}
\def\textqed{\unskip\nobreak\penalty50
       \hskip2em\hbox{}\nobreak\hfil\qedsymbol
       \parfillskip=0pt \finalhyphendemerits=0}
\def\dispqed{\rlap{\qquad\qedsymbol}}

\opn\dis{dis}
\def\pnt{{\raise0.5mm\hbox{\large\bf.}}}

\opn\Lex{Lex}

\begin{document}

\title{Quasi $f$-Simplicial Complexes and Quasi $f$-Graphs  }

\author{ Hasan Mahmood$^{1}$, Fazal Ur Rehman$^{1}$ , Thai Thanh Nguyen$^{3}$, Muhammad Ahsan Binyamin$^{2}$
}
\thanks{\noindent $^{*}$ The first author and the last author are supported by the Higher Education Commission of Pakistan for this research (Grant no. 7515).\\
\noindent $^{1}$Government College University, Lahore, Pakistan. $^{2}$ Government College University, Faisalabad, Pakistan.
$^{3}$ Tulane University, USA; Hue University, College of Education, Vietnam.\\
{\em E-mails }: hasanmahmood@gcu.edu.pk,
fazalqau@gmail.com,  tnguyen11@tulane.edu, ahsanbanyamin@gmail.com}
\maketitle
\begin{abstract}
 The notion of $f$-ideal is recent and has so far been studied in several papers. In \cite{qfi}, the idea of $f$-ideal is generalized to quasi $f$-ideals, which is much larger class than the class of $f$-ideals.  In this paper, we introduce the concept of quasi $f$-simplicial complex and quasi $f$-graph.  We give a characterization of quasi $f$-graphs on $n$ vertices. A complete solution of connectedness of quasi $f$-simplicial complexes is described. We have also shown a method of constructing Cohen-Macaulay quasi $f$-graphs.
 \vskip 0.4 true cm
\noindent
\end{abstract}

\section{Introduction }

Commutative algebra supplies basic methods  in the algebraic study
of combinatorics on convex polytopes and simplicial complex. Richard Stanley
was the first who used in a systematic way concepts and technique
for commutative algebra to study simplicial complex by considering
the Hilbert function of Stanley-Reisner rings, whose defining ideals
are generated by square-free monomials. A square-free monomial ideal
$I$ is an ideal of a polynomial ring $S=k[x_{1},x_{2},...,x_{n}]$ in
$n$ indeterminate over the field $K$ generated by the square-free
monomials. Corresponding to every  square-free monomial
ideal $I$ of $S$, there are two natural simplical complexes, namely, the facet complex of $I$, denoted by $\delta _{\mathcal{F}}(I)$,  and the non-face complex
$\delta_{\mathcal{N}}(I)$.
The equality of $f$-vectors of these two complexes gives us $f$-ideals; whereas the quasi $f$-ideals shows the interconnections and relevance of the $f$-vectors of these two naturally associated complexes to $I$. The notion of $f$-ideals was introduced in 2012 in \cite{deg2}. Later on, the idea of $f$-graphs was
introduced  in \cite{fgraph}. A simple finite graph $G$ on $n$
vertices is an $f$-graph if its edge ideal $I(G)$ is an $f$-ideal of
degree 2. These notions have been studied for it various properties
in the papers \cite{deg2} \cite{degd}, \cite{tswu1}, \cite{tswu2},
\cite{tswuPubl}, \cite{fgraph}, \cite{fsimp}, \cite{fnote}, and
\cite{Adam}. In \cite{qfi}, the authors extended this concept
to  the notion of quasi $f$-ideal which is, in fact, a
generalization of $f$-ideal. It turns out that every $f$-ideal is quasi $f$-ideal but not the converse. Moreover, the class of quasi $f$-ideals is  much bigger class than the class of $f$-ideals. Various
characterizations and construction, and the formula for computing
Hilbert function and Hilbert series of the polynomial ring modulo
quasi $f$-ideals of degree 2 can be found in
\cite{qfi}.
\\

\indent
 This paper is set up in five sections. In second section, we recall some basic concepts, and introduce the term of quasi $f$-simplicial complex and quasi $f$-graph. The section focuses on primary study of $f$-graphs; Theorem 3.4 provides a characterization of quasi $f$-graphs. The fourth section of this paper is devoted to the connectedness of quasi $f$-simplicial complexes. We show that all the quasi $f$-simplicial complexes with dimension greater and equal to 2 are connected, Theorem 4.1. Moreover, Theorem 4.2 classify all connected quasi $f$-graphs. In the last section, we have given a construction of Cohen-Macaulay quasi $f$-graphs.

\section{some fundamental concepts}
\indent Throughout this paper, the character $k$ represents a field, $R$ is a polynomial ring over $k$ in $n$ indeterminate $x_1,x_2,\dots,x_n$, and $G$ will denote a finite simple graph on vertex set $V$ with no isolated vertex.. Let us recall some basic concepts to get familiar with simplicial complexes and square-free monomial ideals. Let $V$ be a non-empty finite set and $\Delta$ be a finite collection of subsets of $V$. Then $\Delta$ is said to be a simplicial complex on $V$ if \\ $(i)$ $\{v\}\in\Delta$ for all $v\in V$ and, \\ $(ii)$ For every subset $E$ of $F\in \Delta$ implies $E\in \Delta$ \\
Here $V$ we call the vertex set of the simplicial complex $\Delta$. Each elements of $\Delta$ is known as face and the maximal faces under $\subseteq$ are known as facets. A subset $F\subset V$ is said to be non-face of $\Delta$ if $F\notin \Delta$ and we denote by ${\mathcal{N}}(\Delta)$, the set of minimal non-face of $\Delta$. The dimension of a face $F$ is defined as $|F|-1$, while the dimension of $\Delta$ is the maximum of the dimensions of all faces of $\Delta$. If $F_1,F_2,\ldots,F_r$ are the facets of $\Delta$, we write simplicial complex as $$\Delta=\langle F_{1}, F_{2},..., F_{r}\rangle$$ to say that $\Delta$ is generated by these $F_i's$. A simplicial complex $\Delta$ is said to be pure if all of its facets have the same dimension.
\begin{Remark}\emph{A finite simple graph is actually a $1$-dimensional simplicial complex, it is usually denoted by $G$. We shall denote by $E(G)$, is the set of all facets of $1$-dimensional simplicial complex have dimension $1$.}
\end{Remark}
 A vector $(f_{0},f_{1},...,f_{d})\in \mathbb{Z}^{d+1}$ is said to be an $f$-vector of a $d$-dimensional simplicial complex $%
\Delta $ if and only if $f_{i}$ is a number of $i$-dimensional faces of $%
\Delta $. It is usually denoted by $f(\Delta)$.

\noindent A simplicial complex $\Delta$ over $V$ is said to be connected if for any two
facets $F$ and $F^{'}$ of $\Delta$, there exists a sequence of
facets $F=F_{0},F_{1},...,F_{r}=F^{'}$ such that $F_{i}\cap
F_{i+1}\neq\phi$, where $0\leq i \leq r-1.$ A simplicial complex is
said to be disconnected if its not connected.\\
\indent In the following definitions we recall the relationship between the
algebraic and combinatorial structures due to R. P. Stanley (see
\cite{Stanley}) and S. Faridi \cite{faridi}.

\begin{Definition} $(${\bf{facet ideal and non-face ideal}}$)$
{\em The facet ideal $I_{\mathcal{F}}(\Delta)\subset R$ of simplicial complex $\Delta =\left\langle F_{1},F_{2},...,F_{r}\right\rangle $, is a square-free monomial ideal generated by the square-free monomials $m_{1},m_{2},...,m_{r}$ such that $m_{i}=\prod\limits_{v_{j}\in F_{i}}x_{j}$, where $i$ is coming from
$\{1,2,...,r\}$. A square-free monomial ideal of $R$ of a simplicial complex $\Delta$, denoted by $I_\Delta$ called non-face ideal (or Stanley-Reisner ideal) if it is generated by the square-free monomials $x_F=\prod\limits_{v_{j}\in F}x_{j}$ where $F\in {\mathcal{N}}(\Delta)$ .i.e. $I_\Delta=(x_{F}:F\in {\mathcal{N}}(\Delta))$}
\end{Definition}

\begin{Definition}$(${\bf{facet complex and non-face complex}}$)$
{\em Let $R=k[x_{1},x_{2},...,x_{n}]$ be a polynomial ring over the field $k$ and $I$ be a square-free monomial ideal of $R$. We use $G(I)$ to denote the unique set of minimal generators of $I$. The facet complex of $I$ is a simplicial complex $$\delta
_{\mathcal{F}}(I)=\{\{v_{i_{1}},v_{i_{2}},...,v_{i_{r}}\}\subseteq V \mid x_{i_{1}}x_{i_{2}}...x_{i_{r}}\in G(I)\}$$ and the non-face complex of $I$ is a simplicial complex $$\delta_{\mathcal{N}}(I)=\{\{v_{i_{1}},v_{i_{2}},...,v_{i_{r}}\}\subseteq V \mid x_{i_{1}}x_{i_{2}}...x_{i_{r}}\notin I\}$$}
\end{Definition}

Now we recall the definition of $f$-ideal.
\begin{Definition}
{\em A square-free monomial ideal $I$ of the polynomial ring $R=k[x_1,x_2\\,\ldots,x_n]$ is said to be an $f$-ideal if $f(\delta_{\mathcal{F}}(I))=f(\delta_{\mathcal{N}}(I))$. A simplicial complex $\Delta$ on $n$ vertices is said to be an $f$-simplicial complex if the facet ideal of $\Delta$ is an $f$-ideal of $R$. A $1$-dimensional $f$-simplicial complex is termed as $f$-graphs.  }
\end{Definition}
We refer the readers to \cite{deg2}, \cite{degd}, \cite{tswu1}, \cite{tswu2},
\cite{tswuPubl}, \cite{fgraph}, \cite{fsimp}, \cite{fnote}, and
\cite{Adam} to know more about $f$-ideals, $f$-graphs and $f$-simplicial complexes. The notion of $f$-ideals was generalized to
 quasi $f$-ideals in \cite{qfi}. The study of quasi $f$-ideals is the study of interconnection between the $f$-vectors of the facet complex and the non-face complex of the ideal. The idea is to read off one vector through the other (see \cite{qfi} for more details).  It is
defined as follows.
 \begin{Definition}
{\em Let $(a_1,a_2,\ldots,a_s)\in \mathbb{Z}^s $. A square-free monomial ideal $I$ in the polynomial ring
$R=k[x_{1},x_{2,}...,x_{n}]$ over the field $k$ is said to be  quasi $f$-ideal
of type $(a_1,a_2,\ldots,a_s)$ if and only if  $f(\delta _{\mathcal{N}}(I))- f(\delta _{\mathcal{F}}(I))=\left(
a_1,a_2,\ldots,a_s\right) $.}
\end{Definition}

\begin{Example}{\label {Exp1}}
\emph{Let
$I=(x_{1}x_{2}x_{4},x_{1}x_{2}x_{5},x_{3}x_{4}x_{5},x_{1}x_{4}x_{5})$
be a pure square-free monomial ideal of degree $3$ in the polynomial
ring $R[x_{1},x_{2},x_{3},x_{4},x_{5}].$ Then the primary
decomposition of $I$ is $I=(x_{1},x_{3})\bigcap (x_{1},x_{4})\bigcap
(x_{1},x_{5})\bigcap (x_{2},x_{4})\bigcap (x_{2},x_{5})\bigcap
(x_{4},x_{5}).$ The facet and the non-face complexes of $I$ are
 $$\delta_{\mathcal{F}}(I)=\langle
\{1,2,4\},\{1,2,5\},\{3,4,5\},\{1,4,5\}\rangle$$
 and
$$\delta _{\mathcal{N}}(I)=\langle
\{1,2,3\},\{1,3,4\},\{1,3,5\},\{2,3,4\},\{2,3,5\},\{2,4,5\}\rangle.$$
Then $f(\delta _{\mathcal{F}}(I))=(5,8,4)$ and $f(\delta
_{\mathcal{N}}(I))=(5,10,6).$ Thus $I$ is a quasi $f$-ideal with
type $(0,2,2).$}
\end{Example}

Now we want to include a natural notion relative to quasi $f$-ideals. There are quasi $f$-simplicial complexes and quasi $f$-graphs. They are given below:

\begin{Definition}
{\em Let $(a_1,a_2,\ldots,a_s)\in \mathbb{Z}^s $; let  $\Delta$ be a
simplicial complex on the vertex set $V=\{v_1,v_2,\ldots,v_n\}$. We
say that $\Delta$ is quasi $f$-simplicial complex of type
$(a_1,a_2,\ldots,a_s)$ if the facet ideal of $\Delta$ is
quasi $f$-ideal of type $(a_1,a_2,\ldots,a_s)$ in the ring $R=k[x_1,x_2,\ldots,x_n]$. It is natural to
call $1$-dimensional quasi $f$-simplicial complex as quasi $f$-graph. Indeed, the type of quasi f-graph would be some ordered pair of integers. }
\end{Definition}

In the following Figure 1 we give the complete list of all non-isomorphic quasi $f$-graphs on $n\leq 6$ vertices with type indicated.
\begin{figure}[h!]
  \includegraphics[width=15 cm]{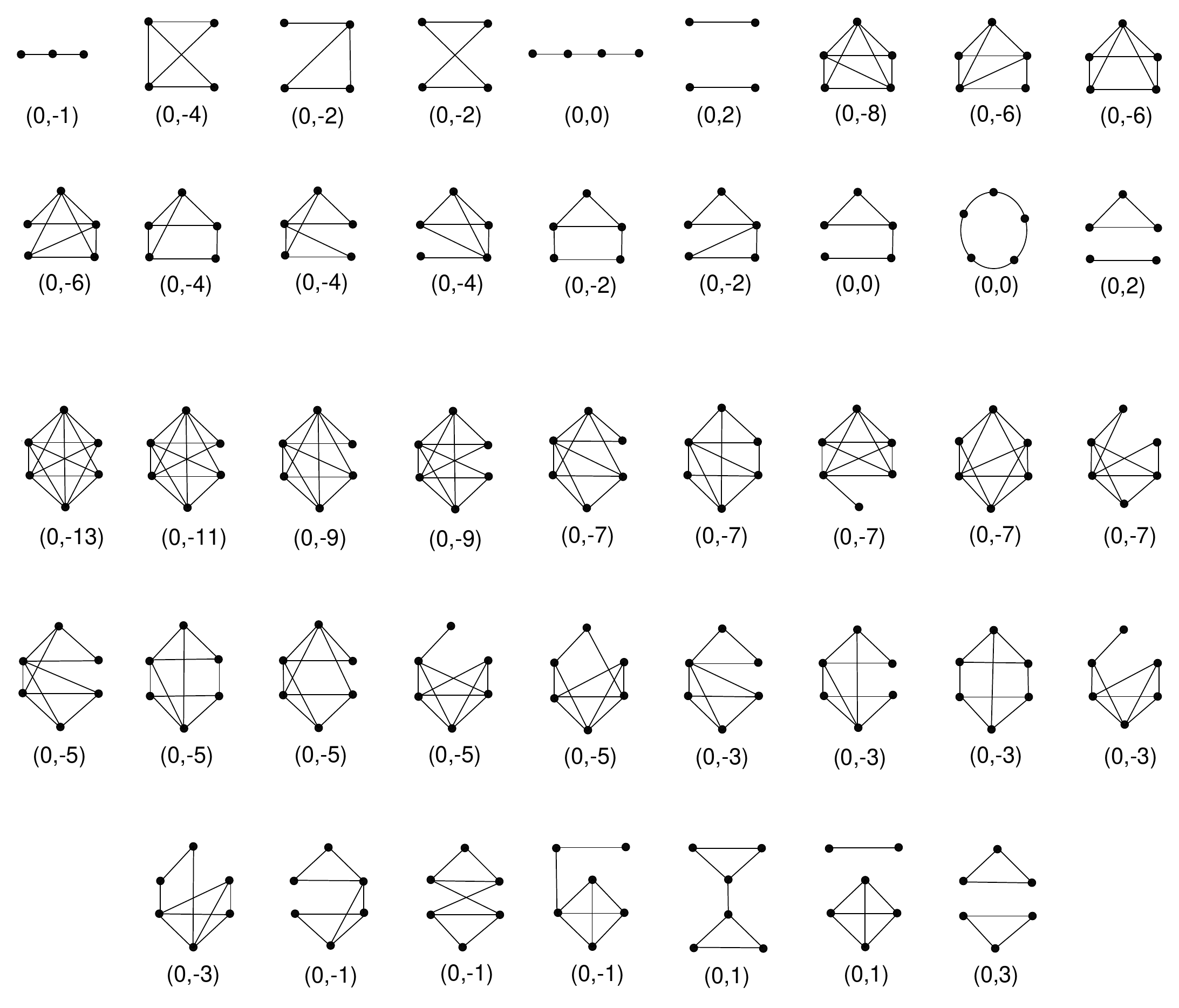}\\
  \caption{ Quasi $f$-graphs up to $n\leq 6$ }\label{y}
\end{figure}

\begin{Remark}
 \emph{It is important to mention that all quasi $f$-graphs will be of the type $(0,b)$, for instance, the type of quasi $f$-graph $G$ on vertex set $V$ will be ordered pair $(a,b)\in\mathbb{Z}^2$. However, since $G$ is a simple graph with no isolated vertex it means that the edge ideal $I(G)$ of $R=k[x_1,x_2,\ldots,x_n]$ is pure square-free monomial quasi $f$-ideal of degree 2 with type $(a,b)\in\mathbb{Z}^2$ and also $supp(G(I(G)))=\{x_1,x_2,\ldots,x_n\}$. Therefore, both the facet complex and the non-face complex of $I(G)$  will have the same vertex set, this means that $a$ must be zero in the ordered pair $(a,b)$. Thus any quasi $f$-graph must be of the type $(0,b)$.}
\end{Remark}

\begin{Example}
{\em Every $f$-simplicial complex ($f$-graph) is a quasi $f$-simplicial (quasi $f$-graph) with type
{\bf{\underline{0}}}-vector. }
\end{Example}

\begin{Example}
{\em In Example \ref{Exp1}, the facet complex of $I$ is a quasi $f$-simplicial complex of type $(0,2,2)$.}
\end{Example}

\begin{Example}
{\em The simplicial complex $\Delta=\langle \{v_1,v_2\},\{v_3,v_4\},\{v_3,v_5\},\{v_1,v_4,v_5\}\rangle$ on $V=\{v_1,v_2,v_3,v_4,v_5\}$ is a non-pure quasi $f$-simplicial complex of type $(0,1,0)$.}
\end{Example}

\begin{Example}
{\em A graph $G=\langle \{v_1,v_2\},\{v_2,v_3\},\{v_3,v_4\},\{v_3,v_5\},\{v_1,v_5\}\rangle$ on $V=\{v_1,v_2,v_3,v_4,v_5\}$ is not a quasi $f$-graph.}
\end{Example}

 \section{Quasi $f$-Graphs and its characterization}
 The purpose of the present section is to give a complete characterization of quasi $f$-graphs. First of all, we would like to recall \cite[Definition 2.1]{tswuPubl} of perfect sets of $R$. Let $sm(R)$ denote the set of all square-free monomials in $R$; let $sm(R)_d$ be the set of all square-free monomials of degree $d$ in $sm(R)$.  For a subset $U\subseteq sm(R)$, we set
$$\sqcup(U)=\{gx_{i} \ | \ g\in U, {x_{i}} \text{  does not divide  } g, 1\leq
i\leq n\}\subset sm(R)_{d+1}$$ and $$\sqcap(U)=\{h \ | \ h=g/x_{i} \text{  for some  } g\in U  \text{  and some  } x_{i} \text{
with  } x_{i}|g\}\subset sm(R)_{d-1}$$
The set $U$ is then called upper perfect if $\sqcup(U)= sm(R)_{d+1}$, and it is said to be lower perfect if $\sqcap(U)= sm(R)_{d-1}$. The set $U$ is called a perfect set if and only if it is both lower and upper perfect. In general, perfect sets can have different cardinalities; for example, every subset of $sm(R)_{d}$ containing a perfect set is again a perfect set. The smallest number among the cardinalities of perfect sets of degree ${d}$ is called the $(n,d)^{th}$ perfect number, and is denoted by
 $N{(n,d)}$. By \cite[Lemma 3.3]{tswuPubl}, for a positive $t$ and $n\geq4$, we have the following equations:

  $$ N(n,2)= \left\{
                                                                                    \begin{array}{ll}
                                                                                      t^2-t, & \hbox{ when $n=2t$;} \\
                                                                                      t^2, & \hbox{ when $n = 2t+1$.}
                                                                                    \end{array}
                                                                                  \right.
$$

The following lemma plays an important role in the characterization of quasi $f$-graphs.
\begin{Lemma}
\emph{Let $G$ be a simple graph on the set of vertices
$\{v_{1},v_{2},...,v_{n}\}$. Then the complementary graph
$\overline{G}$ of $G$ is triangle-free if $G$ is quasi $f$-graph of
type $(0,b).$}
\end{Lemma}
\begin{proof}
If $G$ is a quasi $f$-graph, [Definition 3.1] implies $I(G)$ is a
quasi $f$-ideal. By using \cite[Theorem 4.3]{qfi} the minimal
generating set $G(I(G))$  is upper perfect, this means that
$sm(S)_{3}\subseteq I(G)$. Suppose that $\overline{G}$ contains a triangle of edges $\{{v_{i_1}},{v_{i_2}}\}$, $\{{v_{i_2}},{v_{i_3}}\}$ and $\{{v_{i_3}},{v_{i_1}}\}$. Then it means that there exists a monomial ${x_{i_1}}{x_{i_2}}{x_{i_3}}\notin \sqcup(G(I(G)))$ which is a contradiction.
\end{proof}
It will be interesting to determine the bounds on the values of b, which is given
in the [Proposition 4.4] below.
Here we recall a result from \cite{qfi}.

\begin{Proposition} {\em Let $I$ be a quasi $f$-ideal of degree $2$ and type $(0,b)$ in the polynomial ring $R=k[x_{1},x_{2},...,x_{n}]$. Then the following holds true: $$ -{n\choose 2}+2\leq b\leq {n\choose 2}-2N(n,2) $$}
\end{Proposition}

\begin{Corollary}
\emph{Let $G$ be a quasi $f$-graph of type $(0,b)$ on a vertex set $V=\{v_{1},v_{2},...,v_{n}\}$. Then the bounds of $b$ are follows: $$ -{n\choose 2}+2\leq b\leq {n\choose 2}-2N(n,2) $$}

\end{Corollary}
\begin{proof}
If $G$ is a quasi $f$-graph of type $(0,b)$ on a vertex set $V=\{v_{1},v_{2},...,v_{n}\}$, then it means that $I(G)$ is a quasi $f$-ideal in the polynomial ring $R=k[x_{1},x_{2},...,x_{n}]$ and type $(0,b)$. Using Theorem 3.2 we have the desired inequality.
\end{proof}

Now we give a characterization of quasi $f$-Graphs below.
\begin{Theorem}
\emph{Let $V=\{v_{1},v_{2},...,v_{n}\}$, let $G$ be a simple graph on the
vertex $V$ with no isolated vertices and $|E(G)|={\frac {1}{2}}({n\choose 2}-b)$, where $|b|< {n\choose 2}$.
Then $G$ will be a quasi $f$-graph of type $(0,b)$ if and only if the
complementary graph $\overline{G}$ of $G$ is triangle-free.}
\end{Theorem}


\begin{proof}
Suppose $G$ is a quasi $f$-graph with type $(0,b)$, then [Lemma 3.3]
follows the desired result. For the converse of this theorem,
suppose $\overline{G}$ is a triangle free graph. Therefore,
$sm(S)_{3}\subseteq I(G)$ and this implies $dim (\delta
_{\mathcal{N}}(I(G)))\leq 1$. By using the fact that $|b|<{n\choose
2}$ yields that $dim(\delta _{\mathcal{N}}(I(G)))=1=dim(\delta
_{\mathcal{F}}(I(G))).$ Since $G$ is a simple graph with no isolated
vertices, $supp(I(G))=\{x_{1},x_{2},...,x_{n}\}$ and in view of
\cite[Remark 2.7]{deg2} the both facet complex $(\delta
_{\mathcal{F}}(I(G)))$ and the non face complex $(\delta
_{\mathcal{N}}(I(G)))$ are on same number of vertices, which implies
$f_{0}(\delta _{\mathcal{N}}(I(G)))-f_{0}(\delta
_{\mathcal{F}}(I(G)))=0.$ Note that $(\delta
_{\mathcal{N}}(I(G)))=\overline{G}.$ By \cite[Lemma 3.2]{deg2} we
have $f_{1}(\delta _{\mathcal{N}}(I(G)))={n\choose 2}-f_{1}(\delta
_{\mathcal{F}}(I(G)))$ and as given in above $f_{1}(\delta
_{\mathcal{F}}(I(G)))=|E(G)|={\frac {1}{2}}({n\choose 2}-b)$
together implies $f_{1}(\delta _{\mathcal{N}}(I(G)))-f_{1}(\delta
_{\mathcal{F}}(I(G)))=b.$  The parity of ${n\choose 2}$ is same as
the parity of $b$ implies that ${n\choose 2}\equiv 0\ mod\ 2$ $(1\
mod\ 2)$ if $b$ is even (odd). Hence $I(G)$ is a quasi $f$-ideal and
using [Definition 3.1], $G$ is a quasi $f$-graph of type $(0,b)$
\end{proof}

\section{Connectedness of quasi $f$-simplicial complexes}

 We now concentrate on the problem of the connectedness of quasi $f$-simplicial complexes. In this section, we will classify connected and disconnected quasi $f$-simplicial complexes in terms of their dimensions.

\begin{Theorem}
\emph{Let $V=\{v_{1},v_{2},...,v_{n}\}$ be a vertex set and let $\Delta$
be a pure simplicial complex on $V$ of dimension $d$ with $d>1$. If
$\Delta$ is a quasi $f$-simplicial complex, then $\Delta$ will be
connected.}
\end{Theorem}

\begin{proof}

Suppose $\Delta$ is disconnected quasi $f$-simplicial complex on a vertex set $V$. This means that there are two non-empty disjoint subsets say $V_{1}$ and $V_{2}$ of $V$ such that $V={V_{1}}\bigcup {V_{2}}$ with property that no any facet of $\Delta$ has vertices lie in both $V_1$ and $V_2$. Therefore, We may choose a face $F_{1}\in P(V_{1})$ and another face $F_{2}\in P(V_{2})$ with $dim(F_{1})=d-1$ and $dim(F_{2})=1$ respectively. Then the square-free monomial $x_{F_{1}\bigcup
F_{2}}$ of degree $d+2$ does not belong to $I_{\mathcal{F}}$ $(\Delta),$ which is contradiction to fact that $G(I_{\mathcal{F}}$ $(\Delta))$ is upper perfect.
\end{proof}
The above theorem says that all quasi $f$-simplicial complexes of dimension greater or equal to $2$ are connected. However, for the case of dimension $1$, the situation is different. $1$-dimensional quasi $f$-simplicial complexes may or may not be connected as shown in Figure 2. Now for any graph quasi $f$-graph $G$, it is natural to ask the following questions:\\
\\ (1) When is quasi $f$-graph connected? \\ (2) When is quasi $f$-graph disconnected?\\
 \\
 \indent In next part of this section, we have addressed these questions. However, we need to set some notations and terminologies. Let $m$ and $n$ be two positive integers. A graph $G$ is said to be a $[m:n]$-graph if the complementary graph ${\overline{G}}$ of G is a complete bipartite graph on $m+n$ vertices. i.e. ${\overline{G}}=K_{m,n}$. Note that $[m:n]$-graph $G$ is a disconnected graph having two components $K_m$ and $K_n$, and we can write it as $G=K_{m}\coprod K_{n}=K_{n}\coprod K_{m}.$

\begin{figure}[h!]
  \includegraphics[width=10 cm]{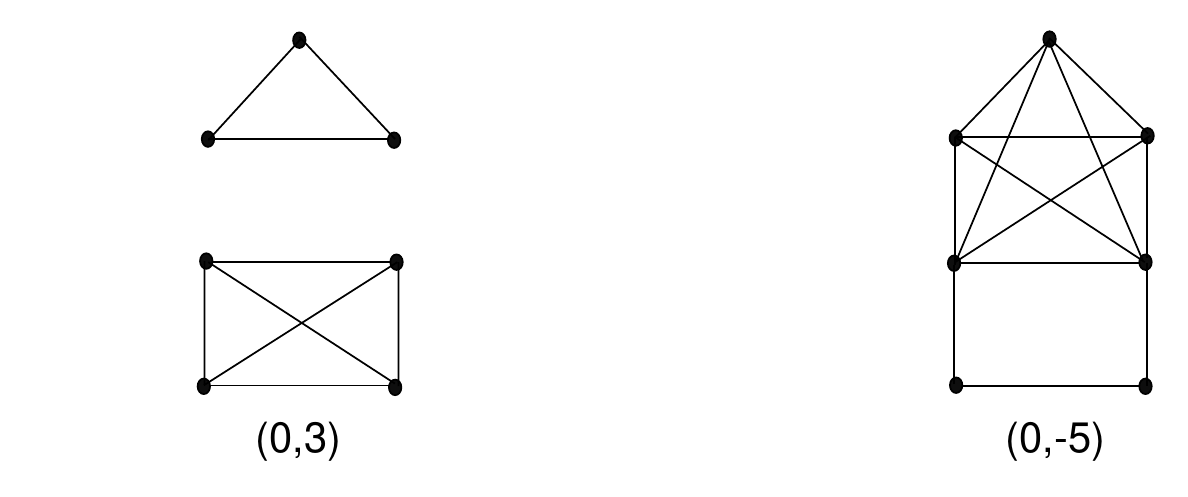}\\
  \caption{ Connected and Disconnected quasi $f$-graphs }\label{y}
\end{figure}

 \begin{Theorem}
\emph{A graph $G$ will be a disconnected quasi $f$-graph of type $(0,b)$
if and only if $G$ is $[m:n]$-graph such that $(m-n)^{2}=m+n-2b.$}
\end{Theorem}

\begin{proof}
If $G$ is a disconnected quasi $f$-graph of type
$(0,b)$, then $G$ would have connected components (say) $G_{1}$ and
$G_{2}.$ Let $m$ and $n$ be positive integers and let $|V(G_{1})|=m$ and $|V(G_{2})|=n$. Obviously, $m,n>1$ since $G$ is a simple graph with on isolated vertices. In order to prove $G$ is a $[m:n]$-graph, it is sufficient to show that $G_{1}=K_{m}$ and $G_{2}=K_{n}$. If $G_{1}$ is not a complete graph on $m$ vertices, then this means $|E(G_{1})|<{m\choose 2}$, which implies that there is at
least one edge exists in the complementary graph of $G_{1}$ (say) $e$ with vertices $v_{i}$ and $v_{j}$. In particular, $e\in E(\overline{G})$. If $v$ is any vertex in $G_{2}$, then the
edges $\{v,v_{i}\},\{v_{i},v_{j}\}$ and $\{v_{j},v\}$ forms a
cycle of length three must contained in $E(\overline{G})$, which contradict to [Lemma 3.1]. Therefore, $G_{1}=K_{m}$. Similarly, $G_{2}=K_{n}$. Next, we want prove that $(m-n)^{2}=m+n-2b$ holds. As we have proved that $G=K_{m}\coprod K_{n}$, this means that $|E(G)|={m\choose 2}+{n\choose 2}$ and as the number of edges of $G$ is ${\frac{1}{2}({m+n\choose 2}-b)}$ since $G$ is a quasi $f$-graph of type $(0,b)$ on $m+n$ vertices, we have the following equation
\begin{equation}
    \label{eq-1}\frac{1}{2}({m+n\choose 2}-b)={m\choose2}+{n\choose2}\end{equation}   It is easy to verify that $E(\overline{G})={{m+n}\choose 2}-E(G)={\frac{1}{2}({m+n\choose 2}+b)}$. As ${\overline{G}}=K_{m,n}$, this means that $E(\overline{G})=mn$ therefore, we have
\begin{equation}
    \frac{1}{2}({m+n\choose 2}+b)=mn\end{equation}

    \begin{equation}
    \label{eq-1} \Rightarrow \frac{1}{2}({m+n\choose 2}-b)=mn-b\end{equation}
Equation (2) and equation (4), together implies $$\Rightarrow {m\choose 2}+{n\choose 2}=mn-b $$ $$\Rightarrow {\frac{m(m-1)}{2}}+{\frac{n(n-1)}{2}}=mn-b$$ $$\Rightarrow m^{2}-m+n^{2}-n=2mn-2b$$ $$\Rightarrow m^{2}+n^{2}-2mn=m+n-2b$$ $$\Rightarrow (m-n)^{2}=m+n-2b$$
\\
\indent Conversely, suppose $G$ is $[m:n]$-graph on $m+n$
vertices such that $(m-n)^{2}=m+n-2b$ holds. Obviously, $G$ is disconnected since $G$ is $[m:n]$-graph. Now we need to prove that $G$ is a quasi $f$-graph. Since $G$ is $[m:n]$-graph, so ${\overline{G}}=K_{m,n}$ this means $\overline{G}$ is a triangle-free graph, because a complete bipartite graph $K_{m,n}$
contains no cycle of odd length. Next, we show that the parity of ${{m+n}\choose 2}$ is same as the parity of $b$ and $|E(G)|={\frac{1}{2}({m+n\choose 2}-b)}$. It is noted that if $G$ is $[m:n]$-graph, then $|E(G)|={m\choose 2}+{n\choose 2}$ and $|E(\overline{G})|=mn$. From relation $(m-n)^{2}=m+n-2b$, we have $$ m^{2}+n^{2}-2mn=m+n-2b$$ $$\Rightarrow m^{2}-m+n^{2}-n=2mn-2b$$ $$\Rightarrow {\frac{m(m-1)}{2}}+{\frac{n(n-1)}{2}}=mn-b$$ $$\Rightarrow {m\choose 2}+{n\choose 2}=mn-b $$ This means that $E(G)=mn-b$. As we know $E(G)+E(\overline{G})={{m+n}\choose 2}$, so we have the following equation
\begin{equation}
    \label{eq-3}2mn-b={{m+n}\choose2}\end{equation} $$\Rightarrow 2mn-2b={{m+n}\choose2}-b$$ $$\Rightarrow E(G)=mn-b={\frac {1}{2}} ({{m+n}\choose2}-b)$$ The equation (5) shows that the parity of ${{m+n}\choose 2}$ is same as the parity of $b$.

\end{proof}

\begin{Corollary}
\emph{A quasi $f$-graph $G$ on a vertex set $V$ of type $(0,b)$ is a connected
if $G$ is not $[m:n]$-graph}
\end{Corollary}

\begin{proof}
If a quasi $f$-graph $G$ is not $[m:n]$-graph, then obviously it is connected.
\end{proof}

\begin{Corollary}
\emph{Let $n$ and $r$ be two positive integers and let $1<r<n$. Then for $n\geq4$, $[n:n-r]$-graph $G$ is disconnected quasi $f$-graph of type $(0,{\frac{1}{2}}(2n-r-r^{2}))$}
\end{Corollary}

\begin{proof}
We need to show that $[n:n-r]$-graph $G$ is disconnected quasi $f$-graph of type $(0,{\frac{1}{2}}(2n-r-r^{2}))$. Let $b={\frac{1}{2}}(2n-r-r^{2})$ and let $m=n-r$. Using [Theorem 4.2] it is sufficient to show that the relation $(n-m)^{2}=m+n-2b$ holds. Let's see $m+n-2b=n-r+n-2{\frac{1}{2}}(2n-r-r^{2})=2n-r-2n+r+r^{2}=r^{2}=(n-m)^{2}$.
\end{proof}

\begin{figure}[h!]
  \includegraphics[width=14 cm]{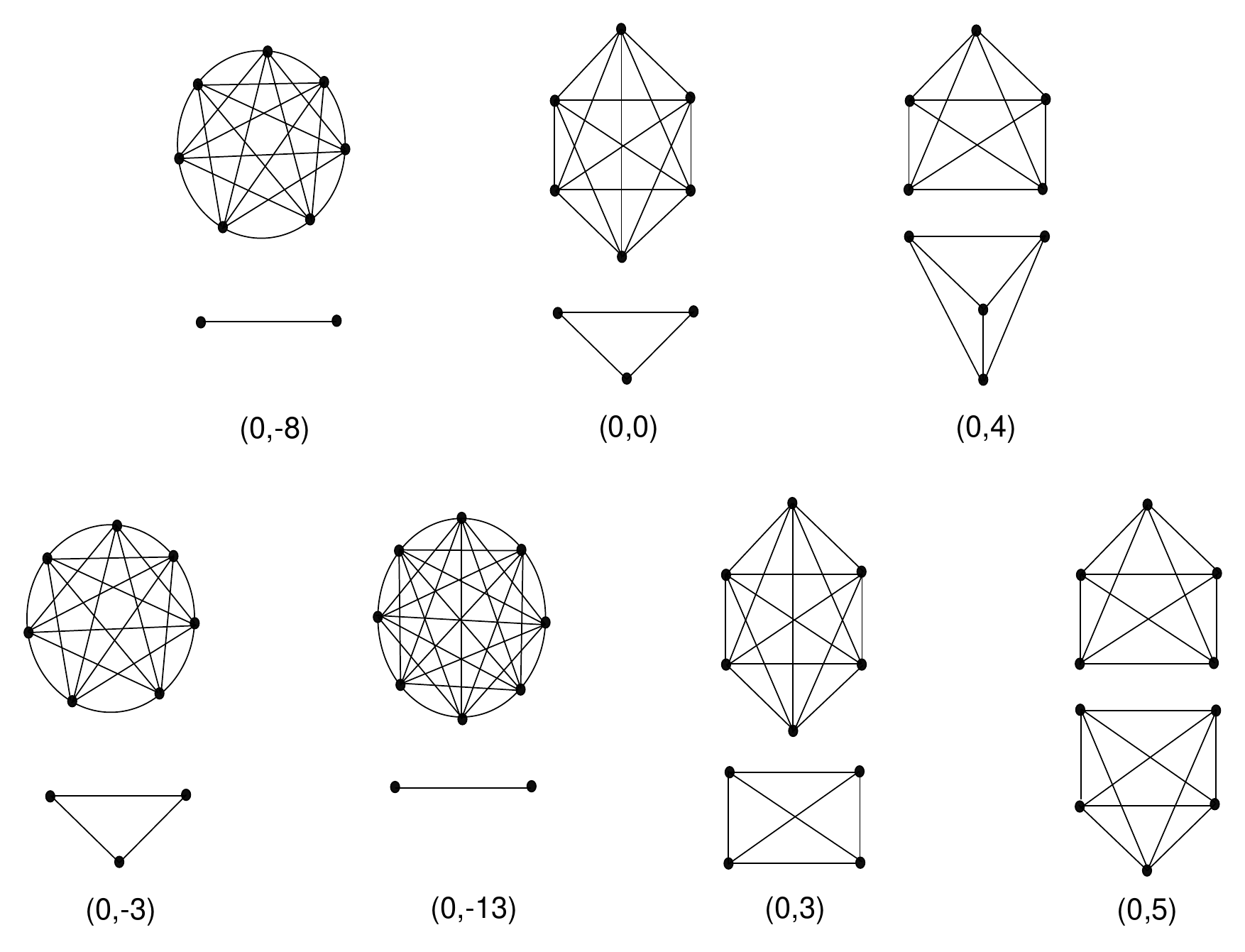}\\
  \caption{ Disconnected quasi $f$-graphs on 9 and 10 vertices }\label{y}
\end{figure}


\section{Construction of Cohen-Macaulay quasi $f$-simplicial complexes}

In this section, we will give a construction of quasi $f$-graphs
which are Cohen-Macaulay. Let us first recall the definition of Cohen-Macaulay Graphs.

\begin{Definition}
\emph{The ring R is called Cohen-Macaulay if its depth is equal to its dimension.}
\end{Definition}

\begin{Definition}
A graph $G$ on the vertex set $V=\{ x_1,x_2,\ldots ,x_n\} $ is said
to be Cohen-Macaulay over the field $k$ if $k[x_1,x_2,\ldots
,x_n]/I(G)$ is a Cohen-Macaulay ring.
\end{Definition}

\begin{Theorem}
\emph{Let $b$ is an integer such that $|b| < \lfloor \frac{n}{2} \rfloor$
and $G$ be a graph on $n$ vertices which is constructed by following
cases:
    \begin{enumerate}
        \item If $n=4k, b=2b'$, $G$ consists of two components $G_1$ and $G_2$ joined with $k-b'$ egdes, where both $G_1$ and $G_2$ are complete graphs on $2k$ vertices.
        \item If $n=4k+1, b=2b'$, $G$ consists of two components $G_1$ and $G_2$ joined with $k-b'$ egdes, where $G_1$ and $G_2$ are complete graphs on $2k+1$ vertices and $2k$ vertices, respectively.
        \item If $n=4k+2, b=2b'+1$, $G$ consists of two components $G_1$ and $G_2$ joined with $k-b'$ egdes, where both $G_1$ and $G_2$ are complete graphs on $2k+1$ vertices.
        \item If $n=4k+3, b=2b'+1$, $G$ consists of two components $G_1$ and $G_2$ joined with $k-b'$ egdes, where $G_1$ and $G_2$ are complete graphs on $2k+2$ vertices and $2k+1$ vertices, respectively.
    \end{enumerate}
Then $G$ is a Cohen-Macaulay quasi $f$-graph of type $(0,b)$.}
\end{Theorem}

\begin{proof}

The condition $|b| < \lfloor \frac{n}{2} \rfloor$ ensures that in each case, $b'< k$, so there are always a positive number of edges joining $G_1$ and $G_2$.\\

First, we check that the number of edges of $G$ as constructed is $\frac{1}{2}({n\choose 2}-b)$. In fact, for the case $(1)$, the number of edges of $G$ is $$2{{2k} \choose 2}+k-b'=4k^2-k-b'=\frac{1}{2}({4k\choose 2}-b).$$ Similarly, we can check number of edges to be $\frac{1}{2}({n\choose 2}-b)$ for the cases $(2)$, $(3)$ and $(4)$. Thus, it is easy to see from the above construction that $G$ is a quasi $f$-graph of type $(0,b)$ - since the complement of $G$ is a bipartite graph, which does not contain any triangle and it has $\frac{1}{2}({n\choose 2}-b)$ edges.\\

 Let us recall from \cite{v} that if $G$ is a graph on $n$ vertices such that $ht(I(G))=n-2$ then $G$ is Cohen-Macaulay if and only if $\delta_{\mathcal{N}}(I(G))$ is connected.  Thus, it suffices to show that the complement of $G$, which is $\delta_{\mathcal{N}}(I(G))$ (since it has no triangle), is connected. In fact, the main idea is the following: since the number of edges joining $G_1$ and $G_2$ is small compared to the maximal number of possible edges between $G_1$ and $G_2$, so when we take the complement, the number of edges matching the vertex sets of $G_1$ and $G_2$ is large enough to make it connected. We will give the calculation case by case and we will see further that we can take $|b|\leq \lfloor \frac{n}{2} \rfloor$ in the assumption of the theorem with some special exceptions (see remark below).\\

Let $G$ be the graph constructed above. The number of edges of the $\overline{G}$ is $\frac{1}{2}({n\choose 2}+b)$. Suppose that $\overline{G}$ is not connected, i.e, there exists the sets $V_1$ with $x$ vertices from $G_1$ and $V_2$ with $y$ vertices from $G_2$ such that all edges of $\overline{G}$ are edges joining vertices from $V_1$ to $V_2$ and vertices from $V(G_1)-V_1$ to $V(G_2)-V_2$.\\

\textbf{Case 1:} The number of edges of $\overline{G}$ is at most
$xy+(2k-x)(2k-y)$. Without loss of generality, assume that $x\leq k$,
\begin{enumerate}
    \item If $x=0$, then $y\geq 1$. This means that the number of edges of $\overline{G}$ is at most $2k(2k-y)$. Since $-2ky\leq -2k<-k+\frac{b}{2}$, we have $$2k(2k-y)=4k^2-2ky<4k^2-k+\frac{b}{2}=\frac{1}{2}({n\choose 2}+b)$$ which is a contradiction. Note that if $|b|= \lfloor \frac{n}{2} \rfloor$ then the inequality above becomes equality if and only if $x=0,y=1$ and $k=-\frac{b}{2}$.
    \item If $x\geq 1$, then since $2xy \leq 2ky$ and $-2kx <-k +\frac{b}{2}$ it holds that $$xy+(2k-x)(2k-y)=4k^2+2xy-2ky-2kx<4k^2-k+\frac{b}{2}=\frac{1}{2}({n\choose 2}+b)$$ which is a contradiction. Note that if $|b|= \lfloor \frac{n}{2} \rfloor$, then the inequality above becomes equality if and only if $x=y=k=1$ and $b=-2$ or $x=1,y=0$ and $k=-\frac{b}{2}$.
\end{enumerate}

 \textbf{Case 2:} The number of edges of $\overline{G}$ is at most $xy+(2k+1-x)(2k-y)$. Without loss of generality, assume that $x\leq k$,
\begin{enumerate}
    \item If $y=0$, then $x\geq 1$. This means that the number of edges of $\overline{G}$ is at most $2k(2k+1-x)$. Since $-2kx <-k+\frac{b}{2}$, we have $$2k(2k+1-x)=4k^2+2k-2kx<4k^2+k+\frac{b}{2}= \frac{1}{2}({n\choose 2}+b)$$
    which is a contradiction. Note that if $|b|= \lfloor \frac{n}{2} \rfloor$ then the inequality above becomes equality if and only if $y=0,x=1$ and $k=-\frac{b}{2}$.
    \item If $y\geq 1$, then since $2xy\leq 2ky$, $-2kx-y<-k+\frac{b}{2}$ (even when $|b|= \lfloor \frac{n}{2} \rfloor$) we have $$xy+(2k+1-x)(2k-y)<4k^2+k+\frac{b}{2}=\frac{1}{2}({n\choose 2}+b)$$
    which is a contradiction.
\end{enumerate}

\textbf{Case 3:} The number of edges of $\overline{G}$ is at most
$xy+(2k+1-x)(2k+1-y)$. Without loss of generality, assume that $x\leq k$,
\begin{enumerate}
    \item If $x=0$ then $y\geq 1$. Since $-2ky\leq -2k< -k +\frac{1}{2}+\frac{b}{2}$, we have
    $$(2k+1)(2k+1-y)<4k^2+3k+\frac{1}{2}+\frac{b}{2}= \frac{1}{2}({n\choose 2}+b)$$ which is a contradiction. Also, if $|b|= \lfloor \frac{n}{2} \rfloor$ then the inequality above becomes equality if and only if $x=0,y=1$ and $b=-(2k+1)$.
    \item If $x\geq 1$, then since $2xy-2ky\leq 0$, $1-x-y \leq 0$ and $-2kx<-k+\frac{1}{2}+\frac{b}{2}$ we have $$xy+(2k+1-x)(2k+1-y)<4k^2+3k+\frac{1}{2}+\frac{b}{2}= \frac{1}{2}({n\choose 2}+b)$$ (contradiction). Also, if $|b|= \lfloor \frac{n}{2} \rfloor$, then the inequality above becomes equality if and only if $x=1,y=0$ and $b=-(2k+1)$.
\end{enumerate}

\textbf{Case 4:} The number of edges of $\overline{G}$ is at most
$xy+(2k+2-x)(2k+1-y)$. Without loss of generality, assume that $x\leq k+1$,
\begin{enumerate}
    \item If $y=0$, then $x\geq 1$. Since $-2kx<-k+\frac{1}{2}+\frac{b}{2}$, we have
    $$(2k+2-x)(2k+1)<4k^2+5k+\frac{3}{2}+\frac{b}{2}=\frac{1}{2}({n\choose 2}+b)$$ (contradiction). Note that if $|b|= \lfloor \frac{n}{2} \rfloor$ then the inequality above becomes equality if and only if $y=0,x=1$ and $b=-(2k+1)$.
    \item If $y\geq 1$, then since $2xy-(2k+2)y\leq 0$ and $-2kx<-k+\frac{1}{2}+\frac{b}{2}$ we have
    $$(2k+2-x)(2k+1-y)<4k^2+5k+\frac{3}{2}+\frac{b}{2}=\frac{1}{2}({n\choose 2}+b)$$ which is a contradiction.
\end{enumerate}
\end{proof}

\begin{Remark}
\emph{As in the proof above, if the assumption of the theorem was $|b|\leq
\lfloor \frac{n}{2} \rfloor$ then the construction still gives us
Cohen-Macaulay graphs except the following cases when $|b|= \lfloor
\frac{n}{2} \rfloor$:
\begin{enumerate}
    \item The graph $C_4$ is of type $(0,-2)$. (When $x=y=k=1$ and $b=-2$)
    \item The graph $K_{2k} \coprod K_{2k}$ or $K_{2k+1} \coprod K_{2k}$ with $2k$ edges joining $1$ vertex from the first component to all vertices of the second component. These are of type $(0,-2k)$.
    \item The graph $K_{2k+1} \coprod K_{2k+1}$ or $K_{2k+2} \coprod K_{2k+1}$ with $2k+1$ edges joining $1$ vertex from the first component to all vertices of the second component. These are of type $(0,-2k-1)$.
\end{enumerate}}
\end{Remark}

\begin{Example}
\emph{At extreme case when $|b|= \lfloor \frac{n}{2} \rfloor$, the graph
$K_2\coprod K_2$ is Cohen-Macaulay quasi $f$-graph of type $(0,2)$
whereas the graph $C_4$ is non Cohen-Macaulay of type $(0,-2)$.}
\end{Example}

\begin{Example}\emph{
(a)Take $n=7$ with $k=1.$ Let us take $b=1$ this means
that $b'=0.$ The quasi $f$-graph will be obtained by joining the
graphs $G_{1}=K_3$ with $G_{2}=K_4$ by 1 edge as shown in figure.\\
(b) Take $n=7$ with $k=1.$ Let us take $b=-1$ this means that
$b'=-1.$ The quasi $f$-graph will be obtained by joining the graphs
$G_{1}=K_3$ with $G_{2}=K_4$ by 2 edge as shown in figure.\\
(c) Take $n=8$ with $k=2.$ Let us take $b=2$ this means
that $b'=1.$ The quasi $f$-graph will be obtained by joining the
graphs $G_{1}=K_4$ with $G_{2}=K_4$ by 1 edge as shown in figure.\\
(d) Take $n=8$ with $k=2.$ Let us take $b=-2$ this means that
$b'=-1.$ The quasi $f$-graph will be obtained by joining the graphs
$G_{1}=K_4$ with $G_{2}=K_4$ by 3 edge as shown in figure.\\
}
\end{Example}
\begin{figure}[h!]
  \includegraphics[width=16 cm]{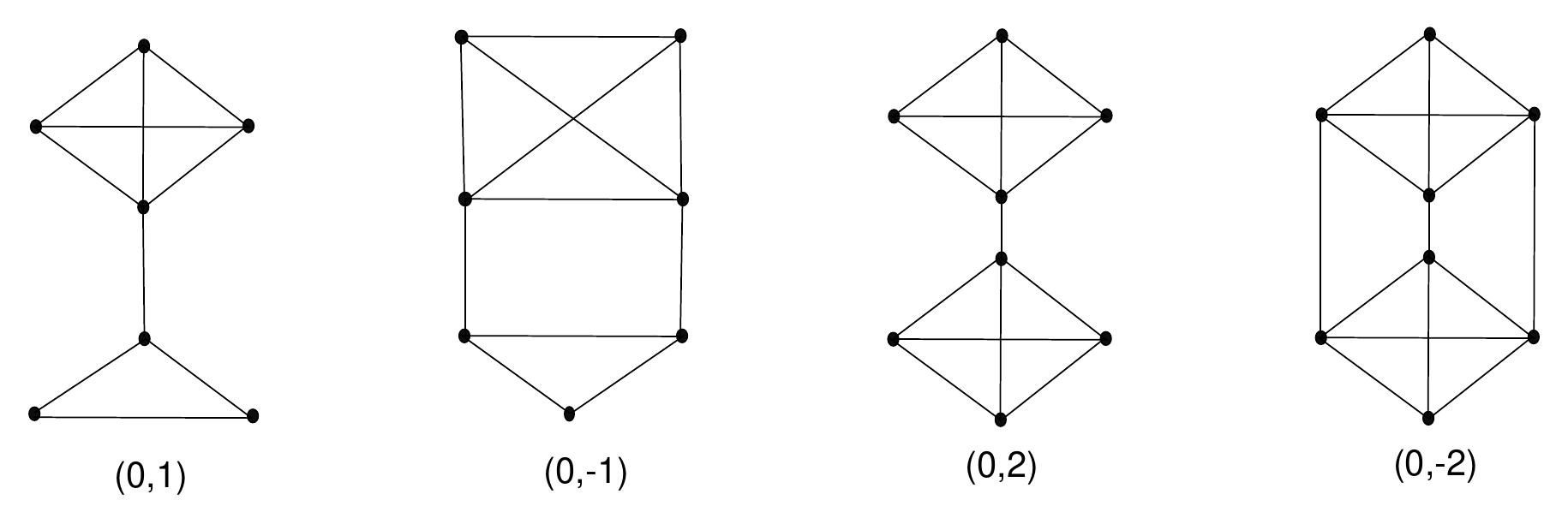}\\
  \caption{ Cohen-Macaulay Quasi $f$-graphs }\label{y}
\end{figure}

Unlike the fact that all $f$-graphs are Cohen-Macaulay, we have lot
of examples of non-Cohen-Macaulay $f$-graphs. Some simple examples
are described in remark above. In particular, among $5$ graphs in
$4$ vertices of quasi $f$-graphs as in figure $1$, two of them are
Cohen-Macaulay and three of them are not. Also, there are a lot more
Cohen-Macaulay quasi $f$-graphs even in small case that is not
constructed by argument above, for example $K_4 \coprod K_2$ (of
type $(0,1)$). It would be interesting to characterize all
Cohen-Macaulay quasi $f$-graphs in particular types.

\end{document}